\documentclass{amsart}
\usepackage{graphicx}
\usepackage{verbatim}

\usepackage{amsmath, amssymb, amsthm}
\usepackage{amsfonts}
\newtheorem{theorem}{Theorem}
\newtheorem{lemma}[theorem]{Lemma}
\newtheorem{definition}[theorem]{Definition}
\numberwithin{equation}{section}
\newtheorem{proposition}[theorem]{Proposition}

\newtheorem*{acknowledgements}{Acknowledgements}

\begin{document}

\title{Labelled version of the almost bounded case of S.\ B.\ Rao's degree sequence conjecture}
\author{Vaidy Sivaraman}

\address{Department of mathematical sciences, Binghamton University.}
\email{vaidy@math.binghamton.edu}

\begin{abstract}
 S.\  B.\  Rao conjectured that graphic sequences are well-quasi-ordered under an inclusion based on induced subgraphs. This conjecture has now been proved by Chudnovsky and Seymour. We give an independent short proof of the labelled version of the almost bounded case of S.\ B.\ Rao's conjecture, the case where we have a bound on the degree, but allow a bounded number of vertices to have unbounded degree.  
\end{abstract}

\keywords{Degree Sequence; S.\ B.\ Rao's conjecture; Well-quasi-ordering}
\date{June 14, 2013 \\  \text{      }  \text{  } \small {2010 Mathematics Subject Classification: 05C07}}   
\maketitle

\noindent
All our graphs are finite and simple, that is, we allow neither loops nor multiple edges. An  integer sequence $(d_1, \ldots ,d_n)$ with  $d_1 \geq  \ldots  \geq d_n \geq 1$ is {\it graphic} if there exists a graph $G$ whose vertex degrees are $d_1, \ldots ,d_n$. A graph $G$ is a {\it realization} of an integer sequence $D$ if the degree sequence of $G$ is $D$. \\

\noindent 
A quasi-order is a pair $(Q, \leq)$, where $Q$ is a set and $\leq$ is a reflexive and transitive relation on $Q$. An infinite sequence $q_1,q_2, \ldots $ in $(Q, \leq)$ is \textit{good} if there exist $i < j$ such that $q_i \leq q_j$. An infinite sequence $q_1,q_2, \ldots $ in $(Q, \leq)$ is \textit{perfect} if there exist $i_1 < i_2 < \ldots $ such that $q_{i_1}  \leq q_{i_2} \leq \ldots$. A quasi-order $(Q, \leq)$ is finite if the underlying set $Q$ is finite. A quasi-order $(Q, \leq)$ is a well-quasi-order (WQO) if every infinite sequence is good. It is not hard to show that every infinite sequence in a WQO is perfect. For a very readable introduction to well-quasi-ordering, we point to  \cite{DI}. \\

\noindent 
Let $(Q, \leq)$ be a WQO.  A sequence $((d_1,q_1), \ldots , (d_n,q_n))$ where $(d_1, \ldots , d_n)$ is a graphic sequence and $q_i \in Q$ for $1 \leq i \leq n$ is called a $Q$-labelled graphic sequence.
A $Q$-labelled graph is a pair $(G, f)$ where $G$ is a graph and $f : V(G) \to Q$ is a function. 
A $Q$-labelled graph $G$ is an induced subgraph of a $Q$-labelled graph $H$ if there exists an injection $\phi : V(G) \to V(H)$ such that (i) $xy \in E(G)$ if and only if $\phi(x) \phi(y) \in E(H)$ for all $x,y \in V(G)$ and (ii) $f_G(x) \leq f_H(\phi(x))$ for all $x \in V(G)$. 
A $Q$-labelled graph $G$ is said to realize a $Q$-labelled sequence $((d_1,q_1), \ldots , (d_n,q_n))$ if $V(G) = \{x_1, \ldots , x_n\}$ and degree of $x_i$ is $d_i$ and $f(x_i) = q_i$ for $1 \leq i \leq n$. \\

\begin{definition}
When $D_1$ and $D_2$ are $Q$-labelled graphic sequences, we write $D_1 \leq D_2$ to mean there exist $Q$-labelled graphs $G_1$ and $G_2$ such that $G_1$ realizes $D_1$, $G_2$ realizes $D_2$, and $G_1$ is an induced subgraph of $G_2$.  
\end{definition}

\noindent
The unlabelled version can be thought of as a special case of the labelled version by choosing $Q$ to be a one-element set. 
It is not hard to show that $\leq$ is both reflexive and transitive. S.\ B.\ Rao \cite{SBR} conjectured  that if $D_1, D_2, \ldots $ is an infinite sequence of graphic sequences, then there exist indices $i <  j$ such that $D_i \leq D_j$. S.\ B.\ Rao's conjecture has now been proved  by Chudnovsky and Seymour \cite{MCPS}.  We will be interested in the almost bounded case of the problem. The bounded version of the conjecture is a natural special case. The proof of Chudnovsky and Seymour for the general case uses three ingredients, one of which is the rooted version of the bounded case (cf. Section 6 in \cite{MCPS}). 
Several short proofs of  the bounded case are known (see \cite{CA}, \cite{VS}). In this article we present an argument that  allows labels from a WQO and works even when there are a bounded number of vertices of unbounded degree. 
More precisely, we prove the following. 

\begin{theorem}\label{BoundedSBRAO}
Let $N$ be a fixed positive integer. If $D_1,D_2, \ldots $ is an infinite sequence of $Q$-labelled graphic sequences with at most $N$ entries in any of them exceeding $N$, then there exist indices $i <  j$ such that $D_i \leq D_j$. 
\end{theorem}

\noindent
We will prove Theorem \ref{BoundedSBRAO}, a restricted version of S.\ B.\ Rao's conjecture, by using Higman's Finite Sequences Theorem.  We use the fact that if the number of entries in an integer sequence (with even sum) is much larger than its highest term, then it is necessarily graphic \cite{ZVZV}. 
\begin{proposition} \label{SufficientGraphicSequence}
 A positive $n$-tuple with even sum and largest entry $d$ is graphic if $n \geq d^2$.
\end{proposition}
\noindent
A proof of Proposition \ref{SufficientGraphicSequence} using the Erd\H{o}s-Gallai Condition for an integer sequence to be graphic can be found in \cite{VS}.\\

\noindent
Finite quasi-orders are WQOs. In particular, if $N$ is a positive integer, the set $\{k \in \mathbb{N}: 1 \leq k \leq N\}$ with equality as the relation is a WQO. We will denote this WQO by $[N]$. \\

\noindent
Let $n$ be a positive integer. For $1 \leq i \leq n$, let $(Q_i, \leq_i)$ be a quasi-order. The Cartesian product of the $n$ quasi-orders is $(Q_1 \times \ldots \times Q_n, \leq)$, where $(a_1, \ldots ,a_n) \leq (b_1, \ldots ,b_n)$ if $a_i \leq_i b_i$ for all $1 \leq i \leq n$.
It is not hard to show that if all $(Q_i, \leq_i)$s are WQOs, then so is their Cartesian product. We will refer to this fact as the Cartesian product theorem. This theorem is a very special easy case of an important result due to Higman. 
If $(Q, \leq)$ is a quasi-order, then consider the set of all finite sequences of elements of $Q$ with the following ordering: $(a_1, \ldots , a_n) \leq_H (b_1, \ldots , b_m)$ if there exist $1 \leq i_1 < \ldots < i_n \leq m$ such that for all $1 \leq k \leq n$, $a_k \leq b_{i_k}$. We will denote this quasi-order by $(Q^{<\omega}, \leq_H)$. Higman proved the following beautiful and extremely important theorem.

\begin{theorem}[Higman \cite{GH}]
 If $(Q, \leq)$ is a WQO, then so is $(Q^{<\omega}, \leq_H)$.
\end{theorem}

\noindent
To prove Theorem \ref{BoundedSBRAO}, we need a lemma. 

\begin{lemma}\label{BoundedLabelledSBRAO}
Let $N$ be a fixed positive integer. If $D_1,D_2, \ldots $ is an infinite sequence of $Q$-labelled graphic sequences with no entry in any of them exceeding $N$, then there exist indices $i <  j$ such that $D_i \leq D_j$. 
\end{lemma}

\begin{proof}
 By the Cartesian product theorem, $[N] \times Q$ is a WQO. Applying Higman's theorem we see that the set of finite sequences with elements in $[N] \times Q$ is a WQO under Higman embedding. Hence there exists an increasing subsequence $C_1, C_2, \ldots$ of $D_1, D_2, \ldots$. We restrict to that. If the length of the sequences $\{C_i\}$s is bounded, then we clearly have an inclusion. If not, let $C_j$ be such that the difference in the length of $C_j$ and $C_1$ is at least $N^2$. By using Proposition \ref{SufficientGraphicSequence} we conclude $C_1 \leq C_j$. 
\end{proof}

\noindent
We are now ready to prove Theorem \ref{BoundedSBRAO}.

\begin{proof}[\bf{Proof of Theorem \ref{BoundedSBRAO}}]
Let $G_i$ be a graph realizing $D_i$. We may assume that each $G_i$ has at least $N$ vertices. Let us rename the $N$ largest degree vertices of each $G_i$ as $x_1, \ldots ,x_N$. Since there are only finitely many non-isomorphic simple graphs on $N$ vertices, we may assume that $G_i | \{x_1, \ldots ,x_N\}$ is the same for all i. By the Cartesian product theorem, we may also assume that the labels for each of the vertices $x_i$ is increasing. The set of subsets of $\{x_1, \ldots ,x_N\}$ with set equality as the relation is a WQO, and we will denote this by $2^N$. Now consider the $(Q \times 2^N)$-labelled graphic sequence $D_1', D_2', \ldots$ where the degree sequence of $D_i'$ is that of $G_i - \{x_1, \ldots ,x_N\}$, and the label being a pair, where the first entry is just the label of the vertex, and the second one is the subset of $\{x_1, \ldots , x_N\}$ to which the vertex is adjacent. By Lemma \ref{BoundedLabelledSBRAO}, we can find indices $i < j$ with $D_i' \leq D_j'$. Putting back the vertices $\{x_1, \ldots , x_N \}$, we see that $H_i \leq H_j$, where $H_i$ is a realization of $D_i$ and $H_j$ is a relaization of $D_j$, giving us the required inclusion $D_i \leq D_j$. 
\end{proof}

\noindent
Higman's theorem is undoubtedly the most important basic result in  the theory of well-quasi-ordering. It inspired Kruskal to prove his famous tree theorem, and motivated Robertson and Seymour to prove their celebrated graph minor theorem. For more on the history of the theory of well-quasi-ordering, the reader is directed to a wonderful survey by Kruskal \cite{JK}. In \cite{VS}, it is shown that Higman's theorem is not needed to prove  the bounded case of S. B. Rao's conjecture. In this article we have demonstrated that the almost bounded case follows easily from Higman's theorem. The proof of S. B. Rao's conjecture given in \cite{MCPS} uses Higman's theorem, but it is not the only main ingredient.  It is not clear whether the general case can be obtained from Higman's theorem in a few easy steps.

\begin{acknowledgements}
I am grateful to Neil Robertson who mentioned the almost bounded case as the next stage after the bounded case.  
\end{acknowledgements}

\end{document}